\documentclass[preprint,12pt]{elsarticle}

\usepackage{amsmath,amssymb,amsfonts,amsthm}

\usepackage{graphicx}

\usepackage{mathrsfs}
\usepackage{mathabx}
\usepackage{dsfont}

\usepackage{color}
\usepackage{cases}

\newtheorem{assumption}{Assumption}
\newtheorem{remark}{Remark}
\newtheorem{theorem}{Theorem}
\newtheorem{lemma}{Lemma}

\let\b=\boldsymbol
\let\mr=\mathrm

\begin{document}

\begin{frontmatter}



\title{Estimations of the discrete Green's function of the  SDFEM  on Shishkin triangular meshes for  problems with only exponential layers}

\author[label1] {Jin Zhang\corref{cor1}}
\cortext[cor1] {jinzhangalex@hotmail.com }
\address[label1]{School of Mathematics and Statistics, Shandong Normal University,
Jinan 250014, China}

\begin{abstract}
In this technical report, we present estimations of the discrete Green's function of the  streamline diffusion finite element method (SDFEM)  on Shishkin triangular meshes for  problems with only exponential layers.
\end{abstract}

\end{frontmatter}

\section{Continuous problem, Shishkin mesh, SDFEM}
We consider the singularly perturbed boundary value problem
 \begin{equation}\label{eq:model problem}
 \begin{array}{rcl}
-\varepsilon\Delta u+\boldsymbol{b}\cdot \nabla u+cu=f & \mbox{in}& \Omega=(0,1)^{2},\\
 u=0 & \mbox{on}& \partial\Omega ,
 \end{array}
 \end{equation}
where $\varepsilon\ll |\boldsymbol{b}|$ is a small positive parameter, 
$\boldsymbol{b}=(b_{1},b_{2})^{T}$ is a constant vector with $b_1>0,b_2>0$ and $c>0$  is constant. It is also
assumed that $f$ is sufficiently smooth. The solution of \eqref{eq:model problem} typically has
two exponential layers of width $O(\varepsilon\ln(1/\varepsilon))$ at the sides
$x=1$ and $y=1$ of $\Omega$.

When discretizing \eqref{eq:model problem}, we use \textit{Shishkin} meshes, which are piecewise uniform. See \cite{Roos:1998-Layer,Roo1Sty2Tob3:2008-Robust,Linb:2003-Layer} for a detailed discussion of their properties and applications.

First, we define two mesh transition parameters, which are    used to specify  mesh changes from coarse to fine in $x-$ and $y-$direction,
\begin{equation*}
\lambda_{x}:=\min\left\{ \frac{1}{2},\rho\frac{\varepsilon}{\beta_{1}}\ln N \right\} \quad \mbox{and} \quad
\lambda_{y}:=\min\left\{
\frac{1}{2},\rho\frac{\varepsilon}{\beta_{2}}\ln N \right\}.
\end{equation*}
 For technical reasons, we set $\rho=2.5$ in our analysis which is the same with ones in \cite{Zhang:2003-Finite} and \cite{Styn1Tobi2:2003-SDFEM}.
We divide $\Omega$ as in Fig. \ref{Shishkin mesh}:  $\overline{\Omega}=\Omega_{s}\cup\Omega_{x}\cup\Omega_{y}\cup\Omega_{xy}$, where
\begin{align*}
&\Omega_{s}:=\left[0,1-\lambda_{x}\right]\times\left[0,1-\lambda_{y}\right],&&
\Omega_{x}:=\left[ 1-\lambda_{x},1 \right]\times\left[0,1-\lambda_{y}\right],\\
&\Omega_{y}:=\left[0,1-\lambda_{x}\right]\times\left[1-\lambda_{y},1 \right],&&
\Omega_{xy}:=\left[ 1-\lambda_{x},1 \right]\times\left[1-\lambda_{y},1 \right].
\end{align*}

\begin{assumption}\label{assum:varepsilon}
We assume  that $\varepsilon\le N^{-1}$, as is generally the case in practice. Furthermore we assume that $
\lambda_{x}=\rho\varepsilon\beta^{-1}_{1}\ln N$ and 
$\lambda_{y}=\rho\varepsilon\beta^{-1}_{2}\ln N$
as otherwise $N^{-1}$ is exponentially small compared with $\varepsilon$.
\end{assumption}

{\color{blue}
\begin{figure}
 \centering 
\begin{minipage}[t]{0.5\linewidth}
\centering
\includegraphics[width=2.5in]{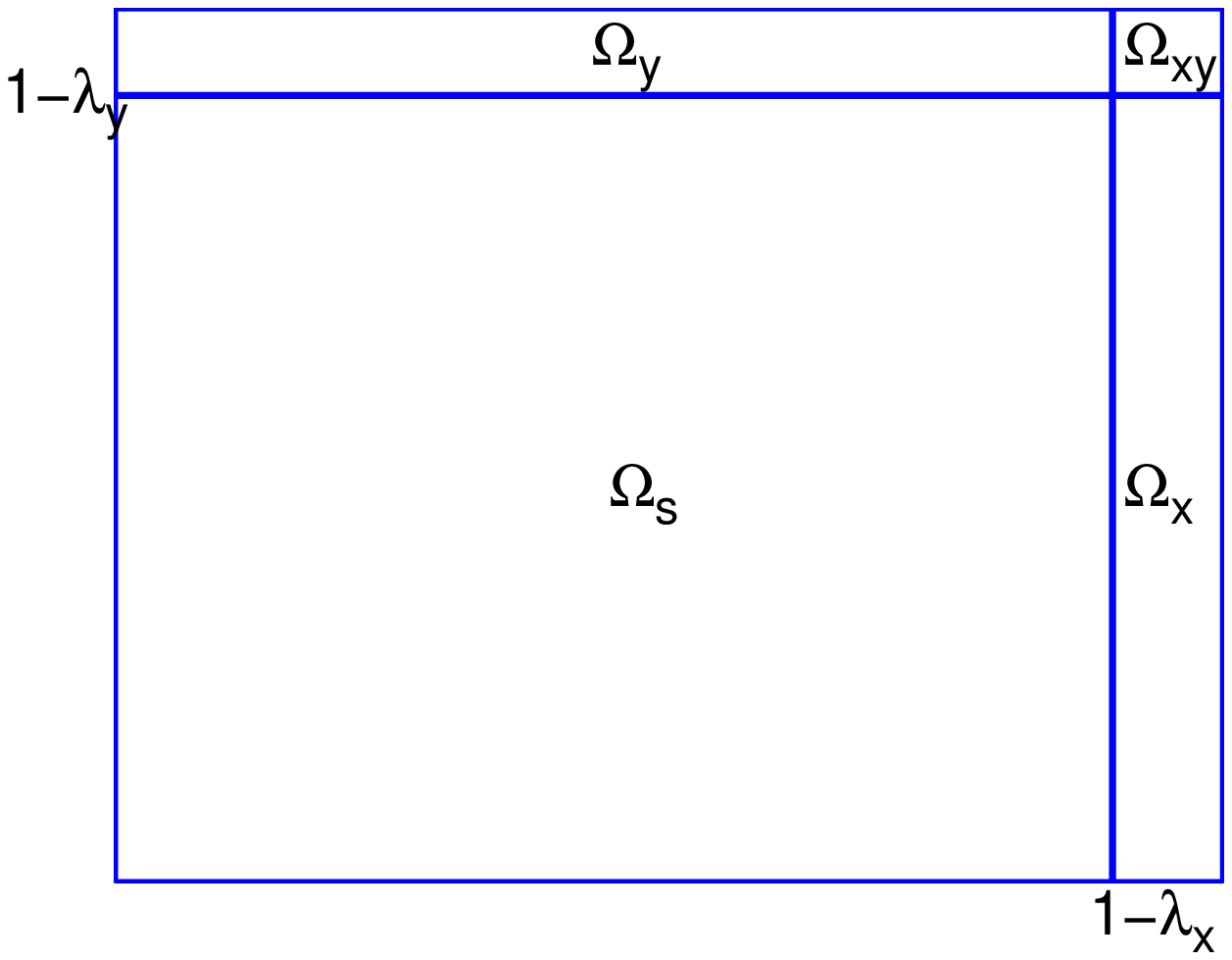}
\end{minipage}%
\begin{minipage}[t]{0.5\linewidth}
\centering
\includegraphics[width=2.5in]{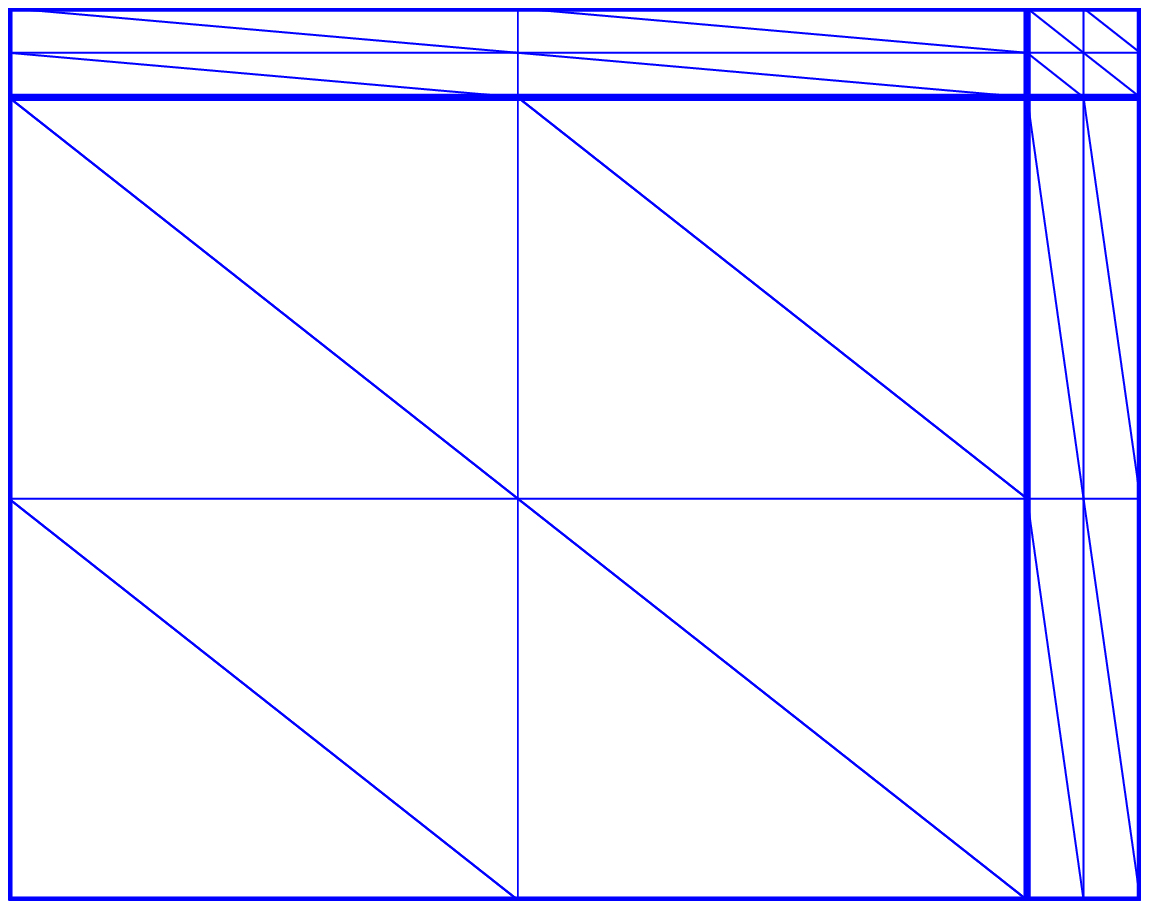}
\end{minipage}
\caption{Dissection of $\Omega$ and triangulation $\mathcal{T}_{N}$.}
\label{Shishkin mesh}
\end{figure}
}

\begin{figure}
\centering
\includegraphics[width=2.5in]{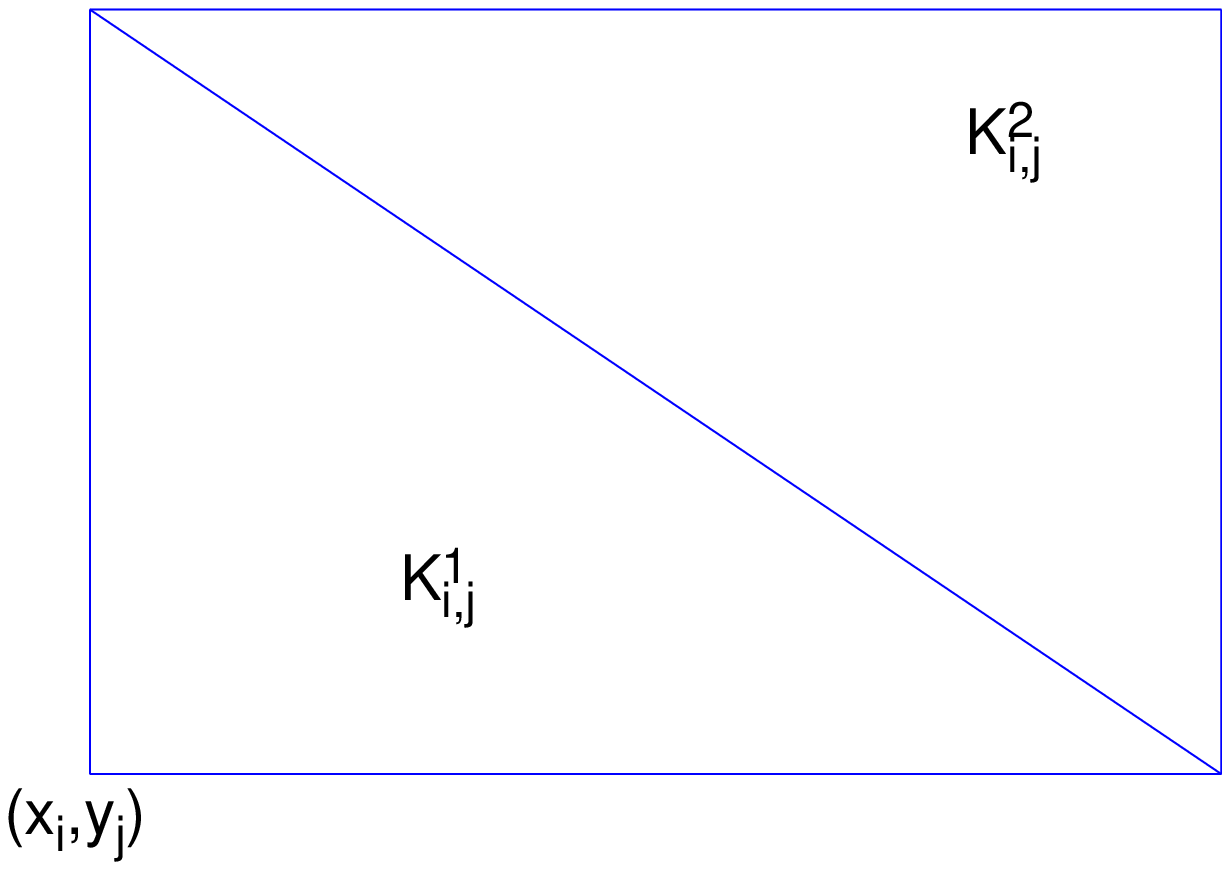}
\caption{$K^{1}_{i,j}$ and $K^{2}_{i,j}$}
\label{fig:code of mesh}
\end{figure}
\par
Next, we define the set of mesh points $\left\{ (x_{i},y_{j})\in\Omega:i,\,j=0,\,\cdots,\,N  \right\}$
\begin{numcases}{x_{i}=}
2i(1-\lambda_{x})/N &\text{for $i=0,\,\cdots,\,N/2$}, \nonumber\\
1-2(N-i)\lambda_{x}/N &\text{for $i=N/2+1,\,\cdots,\,N$}\nonumber
\end{numcases}
and
\begin{numcases}{y_{j}=}
2j(1-\lambda_{y})/N &\text{for $j=0,\,\cdots,\,N/2$}, \nonumber\\
1-2(N-j)\lambda_{y}/N &\text{for $j=N/2+1,\,\cdots,\,N$}.\nonumber
\end{numcases}
By drawing lines through these mesh points parallel to the $x$-axis and $y$-axis the domain $\Omega$ is partitioned into rectangles. Each rectangle is divided into two triangles by drawing the diagonal.
 This yields a triangulation of $\Omega$  denoted by $\mathcal{T}_{N}$(see Fig. \ref{Shishkin mesh}).
The mesh sizes $h_{x,i}:=x_{i+1}-x_{i}$ and $h_{y,j}:=y_{j+1}-y_{j}$ satisfy
\begin{numcases}{h_{x,i}=}
H_{x}:=\frac{1-\lambda_{x}}{N/2}  &\text{for $i=0,\,\cdots,\,N/2-1$}, \nonumber\\
h_{x}:=\frac{\lambda_{x}}{N/2}  &\text{for $i=N/2,\,\cdots,\,N-1$} \nonumber
\end{numcases}
and
\begin{numcases}{h_{y,j}=}
H_{y}:=\frac{1-\lambda_{y}}{N/2}  &\text{for $j=0,\,\cdots,\,N/2-1$},\nonumber \\
h_{y}:=\frac{\lambda_{y}}{N/2}  &\text{for $j=N/2,\,\cdots,\,N-1$}.\nonumber
\end{numcases}
The mesh sizes $H_x,H_y,h_x$ and $h_{y}$ satisfy
\begin{equation*}
N^{-1}\le H_{x},H_{y} \le 2N^{-1} \quad \mbox{and} \quad
C_{1}\varepsilon N^{-1}\ln N \le h_{x},h_{y}\le C_{2}\varepsilon N^{-1}\ln N.
\end{equation*}

For convenience, we shall use the following notations: $K^1_{i,j}$ for the mesh triangle with vertices $(x_i,y_j)$, $(x_{i+1},y_j)$, and $(x_i,y_{j+1})$;  $K^2_{i,j}$  for the mesh triangle with vertices
$(x_i,y_{j+1})$, $(x_{i+1},y_j)$, and $(x_{i+1},y_{j+1})$ (see Fig. \ref{fig:code of mesh}); $K$ for a generic mesh triangle.

On the above Shishkin meshes we define a $C^0$ linear finite element space
\begin{equation*}
V^{N}:=\{v^{N}\in C(\bar{\Omega}):v^{N}|_{\partial\Omega}=0
 \text{ and $v^{N}|_{K}\in P_{1}(K)$, }
  \forall K\in \mathcal{T}_{N} \}.
\end{equation*}

Now we are in a position to state the SDFEM. Let $V:=H^{1}_{0}(\Omega)$ and define
the bilinear forms
\begin{align*}
a_{Gal}(v,w)&=\varepsilon (\nabla v,\nabla w)+(\boldsymbol{b}\cdot\nabla v,w)
+(cv,w) \quad v,w\in V;\\
a_{stab}(v,w)&=\sum_{K\subset\Omega}(-\varepsilon\Delta v+\boldsymbol{b}\cdot\nabla v+cv,\delta_{K}\boldsymbol{b}\cdot\nabla w)_{K} 
\quad v\in V\cap H^2(\mathcal{T}_N),\;w\in V;\\
a_{SD}(v,w)&=a_{Gal}(v,w)+a_{stab}(v,w) \quad v\in V\cap H^2(\mathcal{T}_N),\;w\in V,
\end{align*}
where 
$H^2(\mathcal{T}_N)=\{ v\in L^2(\Omega): \forall K\in \mathcal{T}_N,\; v|_K\in H^2(K)  \}$. The standard SDFEM reads: 
\begin{equation}\label{eq:SDFEM}
\left\{
\begin{array}{lr}
\text{Find $u^{N}\in V^{N}$ such that for all $v^{N}\in V^{N}$},\\
 a_{SD}(u^{N},v^{N})=(f,v^{N})+\underset{K\subset\Omega}\sum(f,\delta_{K}\boldsymbol{b}\cdot\nabla v^{N})_{K}.
\end{array}
\right.
\end{equation}
Note that $\Delta u^N=0$ in $K$ for $u^N\vert_K\in P_1(K)$. Following usual practice
\cite{Roo1Sty2Tob3:2008-Robust},  the parameter $\delta_{K}:=\delta|_K$ is defined as follows:
\begin{equation}\label{eq: delta-K}
\delta_{K}=
\left\{
\begin{array}{cc}
C^{\ast}N^{-1}
&\text{if $K\subset\Omega_{s}$},\\
0&\text{otherwise},
\end{array}
\right.
\end{equation}
and  $C^{\ast}$ is a properly defined positive constant such that the following coercivity holds  (see \cite[Lemma 3.25]{Roo1Sty2Tob3:2008-Robust}):
\begin{equation}\label{eq:coercivity}
a_{SD}(v^{N},v^{N})\ge \frac{1}{2} \vvvert v^{N} \vvvert^2 
\quad \forall v^{N}\in V^{N},
\end{equation}
where
\begin{equation}\label{eq:SD norm}
 \vvvert v^{N}  \vvvert^{2}:=
\varepsilon \vert v^{N} \vert^{2}_{1}+\Vert v^{N} \Vert^{2}
+\sum_{K\subset\Omega} \delta_{K}\Vert
\b{b}\cdot\nabla v^{N}\Vert^{2}_{K}.
\end{equation}
Coercivity \eqref{eq:coercivity} implies a unique solution of the discrete problem \eqref{eq:SDFEM}. Also the Galerkin orthogonality holds, i.e.,
\begin{equation}\label{eq:orthogonality of SD}
a_{SD}(u-u^N,v^N)=0 \quad \forall v^N\in V^N.
\end{equation}

Set
\begin{equation*}
    b:=\sqrt{b^{2}_{1}+b^{2}_{2}},\quad
    \boldsymbol{\beta}:=\genfrac(){0cm}{0}{b_{1}}{b_{2}}/b,\quad
    \boldsymbol{\eta}:=\genfrac(){0cm}{0}{-b_{2}}{b_{1}}/b.
\end{equation*}
For our later analysis, we define a mesh subdomain of $\Omega$ for each mesh node $\boldsymbol{x}^{\ast}=(x^*,y^*)$:
\begin{equation}\label{eq:Omega'0}
\Omega^{\prime}_{0}:=\Omega^{\prime}_{0}(\b{x}^*)=\{
K\in\mathcal{T}_{N}:\mbox{meas}(\Omega_{0}\cap K)\ne0
\},
\end{equation}
where 
\begin{equation}\label{eq:Omega0}
\begin{split}
\Omega_{0}:=\Omega_{0}(\b{x}^*)=\big\{ \boldsymbol{x}=(x,y)\in\Omega:&\;(\boldsymbol{x}-\boldsymbol{x}^{*})\cdot\boldsymbol{\beta}\le \mathscr{K} \sigma_{\beta}\ln N \text{ and } \\
&\vert (\boldsymbol{x}-\boldsymbol{x}^{*})\cdot\boldsymbol{\eta} \vert
\le \mathscr{K} \sigma_{\eta}\ln N \big\} 
\end{split}
\end{equation}
 (see Fig. \ref{Omega-0} and Fig. \ref{Omega'-0})
and
\begin{equation}\label{eq:sigma-beta-eta}
\sigma_{\beta}=kN^{-1}\ln N,\;\;\sigma_{\eta}=kN^{-1/2}.
\end{equation}
We shall choose $k>0$ and $\mathscr{K}>0$ later, which are independent of $N$ and $\varepsilon$. Note that
\begin{equation}\label{eq:area of influence domain}
\mr{meas}(\Omega^{\prime}_{0})\le C \sigma_{\eta} \ln N.
\end{equation}

\begin{figure}
\begin{minipage}[t]{0.5\linewidth}
\centering
\includegraphics[width=2.3in]{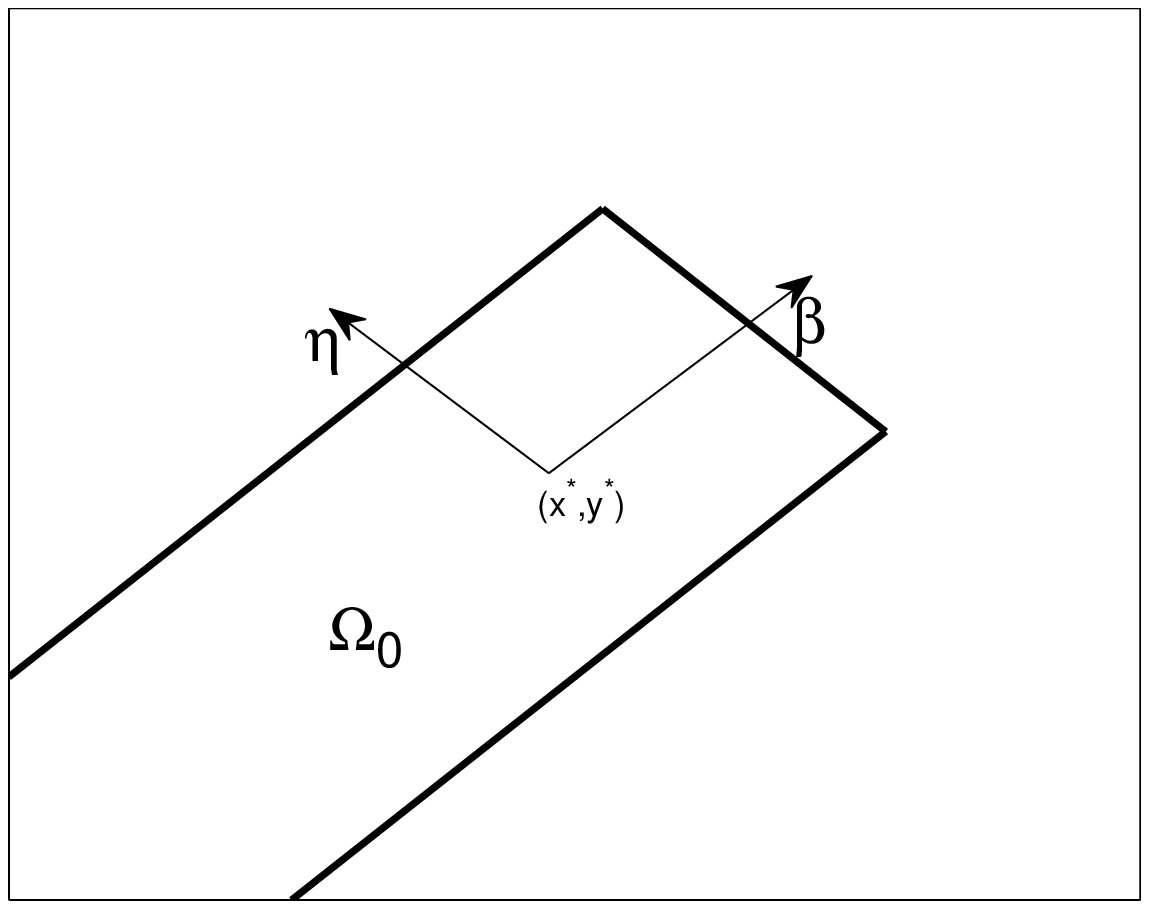}
\caption{Subdomain $\Omega_0=\Omega_0(\b{x}^*)$}
\label{Omega-0}
\end{minipage}%
\begin{minipage}[t]{0.5\linewidth}
\centering
\includegraphics[width=2.4in]{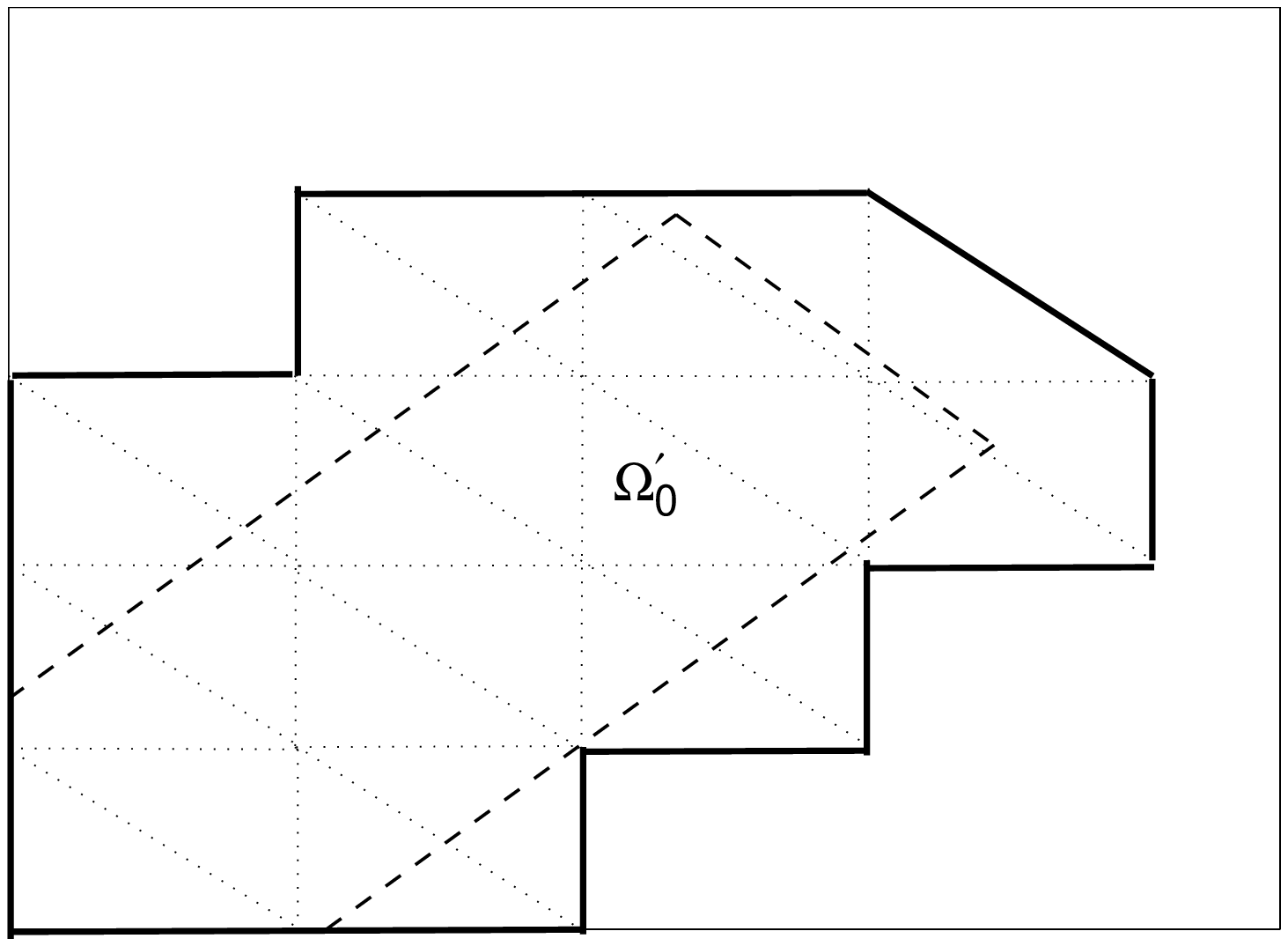}
\caption{Subdomain $\Omega'_0=\Omega'_0(\b{x}^*)$}
\label{Omega'-0}
\end{minipage}
\end{figure}
\section{The discrete Green's function}

In this section, we introduce the discrete Green's  function and cite some results from \cite{Linb1Styn2:2001-SDFEM}.

Let $\boldsymbol{x}^{\ast}=(x^*,y^*)$ be a mesh node in $\Omega$. The discrete Green's function $G\in V^{N}$ associated with $\boldsymbol{x}^{\ast}$ is defined by
\begin{equation}\label{eq:discrete Green function}
a_{SD}(v^{N},G)=v^{N}(\boldsymbol{x}^{\ast})\quad \forall v^{N}\in V^{N}.
\end{equation}

In the following analysis, we present the energy estimation  of discrete Green's functions. Our analysis is similar to that of \cite{Linb1Styn2:2001-SDFEM}, with some changes.

We define a weight function 
\begin{equation*}
\omega(\boldsymbol{x}):=
g\left(\frac{(\boldsymbol{x}-\boldsymbol{x}^{\ast})\cdot\boldsymbol{\beta}}{\sigma_{\beta}}\right)
g\left(\frac{(\boldsymbol{x}-\boldsymbol{x}^{\ast})\cdot\boldsymbol{\eta}}{\sigma_{\eta}}\right)
g\left(-\frac{(\boldsymbol{x}-\boldsymbol{x}^{\ast})\cdot\boldsymbol{\eta}}{\sigma_{\eta}}\right)
\end{equation*}
with $g(r)=2/(1+e^{r})$ for $r\in(-\infty,\infty)$. 
We shall choose $k>0$ later. 
Note Lemma 4.1 in \cite{Linb1Styn2:2001-SDFEM} holds if $\sigma_{\beta}\ge k N^{-1}\ln N$ and $\sigma_{\eta}\ge k N^{-1/2}$.

Now we define a weighted energy norm 
\begin{equation}\label{eq:weighted norm}
\begin{split}
 \vvvert G \vvvert^{2}_{\omega}:= &\varepsilon \Vert \omega^{-1/2}G_{\beta} \Vert^{2}
+\varepsilon\Vert \omega^{-1/2}G_{\eta} \Vert^{2}+\frac{b}{2}\Vert (\omega^{-1})^{1/2}_{\beta} G\Vert^{2}\\
&+c\Vert \omega^{-1/2}G \Vert^{2}
+\sum_K b^{2}\delta_K\Vert \omega^{-1/2}G_{\beta} \Vert^{2}_K.
\end{split}
\end{equation}
Note that $(\omega^{-1})_{\beta}>0$. For any subdomain $D$ of $\Omega$, let
$\vvvert G \vvvert_{\omega,D}$ mean that the integrations in \eqref{eq:weighted norm} are restricted to $D$.
From \eqref{eq:SDFEM}, \eqref{eq:weighted norm} and integration by parts, we have
\begin{equation*}\label{eq:norm and bilinear form}
\begin{split}
\vvvert G \vvvert^{2}_{\omega}
=&a_{SD}(\omega^{-1}G,G)-\varepsilon((\omega^{-1})_{\beta}G,G_{\beta})
- \varepsilon ((\omega^{-1})_{\eta}G,G_{\eta})\\
&
-\sum_K  ( b(\omega^{-1})_{\beta}G+c \omega^{-1}G,\delta_K \; bG_{\beta})_K.
\end{split}
\end{equation*}
Considering \eqref{eq:discrete Green function}, we  have
\begin{equation*}\label{eq:idea-weighted estimate}
\begin{split}
a_{SD}(\omega^{-1}G,G)&=a_{SD}(\omega^{-1}G-(\omega^{-1}G)^{I},G)+a_{SD}((\omega^{-1}G)^{I},G)\\
&=a_{SD}(\omega^{-1}G-(\omega^{-1}G)^{I},G)+(\omega^{-1}G)(\boldsymbol{x}^{\ast}).
\end{split}
\end{equation*}
 By means of the above two equalities, the  energy estimate of $G$ will be obtained from the next three Lemmas.

\begin{lemma}\label{lem:1}
Assume $\sigma_{\beta}\ge kN^{-1} $ and $\sigma_{\eta}\ge k \varepsilon^{1/2}$  in \eqref{eq:Omega0}, then for $k>1$ sufficiently
large and independent of $N$ and $\varepsilon$, we have
\begin{equation*}
a_{SD}(\omega^{-1}G,G)\ge \frac{1}{4}\vvvert G \vvvert^{2}_{\omega}.
\end{equation*}
\end{lemma}
\begin{proof}
See \cite[Lemma 4.2]{Linb1Styn2:2001-SDFEM}.
\end{proof}

\begin{lemma}\label{lem:2}
If $\sigma_{\beta}\ge kN^{-1}$ in \eqref{eq:Omega0}, with $k>0$ independent of $N$ and $\varepsilon$, then for each mesh point $\boldsymbol{x}^{\ast}\in\Omega\setminus\Omega_{xy}$, we have
\begin{equation*}
\left|(\omega^{-1}G)(\boldsymbol{x}^{\ast})\right|\le
\frac{1}{16}\vvvert G \vvvert^{2}_{\omega}
+
\left\{
\begin{matrix}
CN^2\sigma_{\beta}& \text{if $\b{x}^*\in\Omega_s$}\\
CN\ln N& \text{if $\b{x}^*\in\Omega_x\cup\Omega_y$}
\end{matrix}
\right.
.
\end{equation*}
where $C$ is independent of $N$, $\varepsilon$ and $\boldsymbol{x}^{\ast}$.
\end{lemma}
\begin{proof}
See \cite[Lemma 4.3]{Linb1Styn2:2001-SDFEM}.
\end{proof}

\begin{lemma}\label{lem:3}
If $\sigma_{\beta}\ge kN^{-1}\ln N $ and $\sigma_{\eta}\ge k N^{-1/2}$ in \eqref{eq:Omega0}, where $k>1$ is sufficiently
large and independent of $N$ and $\varepsilon$, then
\begin{equation*}
a_{SD}((\omega^{-1}G)^{I}-\omega^{-1}G,G)\le \frac{1}{16}\vvvert G \vvvert^{2}_{\omega}.
\end{equation*}
\end{lemma}
\begin{proof}

For convenience we set
$\tilde{E}(\boldsymbol{x}):=((\omega^{-1}G)^{I}-\omega^{-1}G)(\boldsymbol{x})$.
Recall $b$ is constant and integration by parts yields $(b\tilde{E}_{\beta},G)=-(b\tilde{E},G_{\beta})$. Then we have
\begin{equation}\label{B(E,G)}
\begin{split}
|a_{SD}(\tilde{E},G)|
\le
C\big( &\Vert (\varepsilon+b^{2}\delta)^{1/2}\omega^{1/2}\tilde{E}_{\beta} \Vert
+\varepsilon^{1/2}\Vert \omega^{1/2}\tilde{E}_{\eta} \Vert\\
&+\Vert (\varepsilon+b^{2}\delta)^{-1/2}\omega^{1/2}\tilde{E} \Vert
\big) \vvvert  G \vvvert_{\omega}.
\end{split}
\end{equation}

To analyze different kinds of interpolation bounds, we first estimate the following terms. Note that $G_{\beta\beta}=G_{\eta\eta}=G_{\beta\eta}=0$ on $K$
because $G$ belongs to $ V^{N}$. 
For convenience, we set $M_{K}:=\underset{K}{\max}\,\omega^{-1/2}$.
Using   (iii), (iv) and (v) in \cite[Lemma 4.1]{Linb1Styn2:2001-SDFEM}, we obtain
\begin{equation}\label{eq:omega-G-xx}
\begin{split}
&\Vert (\omega^{-1}G)_{\beta\beta} \Vert_{K}
\le
\Vert(\omega^{-1})_{\beta\beta}G \Vert_{K}+\Vert(\omega^{-1})_{\beta}G_{\beta}\Vert_{K}\\
\le&
CM_{K}
\left(  
\sigma^{-3/2}_{\beta}\Vert (\omega^{-1})^{1/2}_{\beta}G \Vert_{K}+
\sigma^{-1}_{\beta}\Vert \omega^{-1/2} G_{\beta} \Vert_{K}
\right)\\
\le&
CM_{K}
\left( 
\sigma^{-3/2}_{\beta}+
\sigma^{-1}_{\beta}(\varepsilon+b^{2}\delta)^{-1/2}
\right)
\vvvert G \vvvert_{\omega,K}.
\end{split}
\end{equation}
Note
$\Vert     G_{ \eta } \Vert_{K}\le C \max\{ h^{-1}_{x,K}, h^{-1}_{y,K} \} \Vert     G  \Vert_{K}$
or
$\Vert     G_{ \eta } \Vert_{K}\le C \varepsilon^{-1/2}\cdot \varepsilon^{1/2}\Vert     G_{\eta}  \Vert_{K}$,
then we have
$$
\Vert     G_{ \eta } \Vert_{K}\le C\min\{   \max\{ h^{-1}_{x,K}, h^{-1}_{y,K} \}, \varepsilon^{-1/2} \} 
\vvvert G \vvvert_{K},
$$
and
\begin{equation}\label{eq:omega-y-G-y}
\begin{split}
&\Vert  (\omega^{-1})_{  \eta  }G_{  \eta  } \Vert_{K} 
\le 
C\underset{K}{\max} |(\omega^{-1})_{  \eta  }|\;   \Vert     G_{  \eta  } \Vert_{K} \\
\le &
CM_{K}\sigma^{-1}_{  \eta  }  \underset{K}{\max}\omega^{-1/2} \cdot 
\min\{   \max\{ h^{-1}_{x,K}, h^{-1}_{y,K} \}, \varepsilon^{-1/2} \} 
\vvvert G \vvvert_{K}\\
\le &
C M_{K} \sigma^{-1}_{ \eta } \min\{   \max\{ h^{-1}_{x,K}, h^{-1}_{y,K} \}, \varepsilon^{-1/2} \} 
\vvvert G \vvvert_{\omega,K}.
\end{split}
\end{equation}
Similarly,  we have $\Vert(\omega^{-1})_{ \eta  \eta }G \Vert_{K}\le CM_{K} \sigma^{-2}_{ \eta }\Vert \omega^{-1/2}G \Vert_{K}$
and 
\begin{equation}\label{eq:omega-G-yy}
\begin{split}
&\Vert (\omega^{-1}G)_{ \eta  \eta } \Vert_{K}
\le
\Vert(\omega^{-1})_{ \eta  \eta }G \Vert_{K}+\Vert(\omega^{-1})_{ \eta }G_{ \eta }\Vert_{K}\\
\le &
CM_{K}
\left( \sigma^{-2}_{ \eta }+ \sigma^{-1}_{  \eta } \min\{   \max\{ h^{-1}_{x,K}, h^{-1}_{y,K} \}, \varepsilon^{-1/2} \}   \right)
\vvvert  G \vvvert_{\omega,K}.
\end{split}
\end{equation}
Recalling (v) in \cite[Lemma 4.1]{Linb1Styn2:2001-SDFEM} and inverse estimates, we have
\begin{equation}\label{eq:omega-x-G-y-I}
\begin{split}
&\Vert  (\omega^{-1})_{\beta }G_{ \eta}  \Vert_{K}
\le
C \underset{K}{\max} (\omega^{-1})_{\beta }\cdot  \Vert G_{ \eta}  \Vert_{K}\\
\le &
C \underset{K}{\max} (\omega^{-1})_{\beta }\cdot \max\{ h^{-1}_{x,K}, h^{-1}_{y,K} \} \cdot \Vert G  \Vert_{K} \\
\le &
C \max\{ h^{-1}_{x,K}, h^{-1}_{y,K} \}\left( \underset{K}{\max} (\omega^{-1})_{\beta } \right)^{1/2}
\left( \underset{K}{\min} (\omega^{-1})_{\beta } \right)^{1/2}
\cdot     \Vert G  \Vert_{K}\\
\le  &
CM_K \cdot \max\{ h^{-1}_{x,K}, h^{-1}_{y,K} \} \sigma^{-1/2}_{\beta}\cdot     \Vert (\omega^{-1})^{1/2}_{\beta } G  \Vert_{K}.
\end{split}
\end{equation}
Also, we have
\begin{equation}\label{eq:omega-x-G-y-II}
\Vert  (\omega^{-1})_{\beta }G_{ \eta}  \Vert_{K}
\le
CM_{K}\cdot
 \varepsilon^{-1/2}\sigma^{-1}_{ \beta}\cdot  \varepsilon^{1/2}\Vert  \omega^{-1/2}G_{ \eta} \Vert_{K}.
\end{equation}
Then from \eqref{eq:omega-x-G-y-I} and \eqref{eq:omega-x-G-y-II}, we have
\begin{equation*}
\Vert  (\omega^{-1})_{\beta }G_{ \eta}  \Vert_{K}
\le 
C M_{K} \min\{ \max\{ h^{-1}_{x,K}, h^{-1}_{y,K} \} \sigma^{-1/2}_{\beta},\varepsilon^{-1/2}\sigma^{-1}_{ \beta}\} \vvvert G \vvvert_{\omega,K},
\end{equation*}
and then
\begin{equation}\label{eq:omega-G-xy}
\begin{split}
&\Vert (\omega^{-1}G)_{\beta \eta} \Vert_{K}
\le
\Vert(\omega^{-1})_{\beta \eta}G \Vert_{K}+\Vert(\omega^{-1})_{\eta}G_{\beta }\Vert_{K}
+\Vert(\omega^{-1})_{\beta }G_{\eta}\Vert_{K}\\
\le &
CM_{K}
\left( \sigma^{-1/2}_{\beta}\sigma^{-1}_{ \eta }\Vert (\omega^{-1})^{1/2}_{\beta}G \Vert_{K}+
\sigma^{-1}_{ \eta }\Vert \omega^{-1/2} G_{\beta} \Vert_{K}
\right)+\Vert(\omega^{-1})_{\beta}G_{ \eta }\Vert_{K}\\
\le &
CM_{K}
\big(  \sigma^{-1/2}_{\beta}\sigma^{-1}_{ \eta }+
\sigma^{-1}_{ \eta }(\varepsilon+b^{2}\delta)^{-1/2}\\
&+ \min\{ \max\{ h^{-1}_{x,K}, h^{-1}_{y,K} \} \sigma^{-1/2}_{\beta},\varepsilon^{-1/2}\sigma^{-1}_{ \beta}\}  \big)
\vvvert  G \vvvert_{\omega,K}.
\end{split}
\end{equation}

Set  $h_{K}=\max\{h_{x,K},h_{y,K}\}$ and $\tilde{M}_K:= (\underset{K}{\max} \omega)^{1/2}$. From \cite[Corollary 3.1]{Linb1Styn2:2001-SDFEM}, we have
\begin{equation}\label{interpolation of Ex and Ey}
\begin{split}
&\Vert \omega^{1/2} \tilde{E} \Vert_{K}+h_{K}\Vert \omega^{1/2} \tilde{E}_{\beta} \Vert_{K}+h_{K}\Vert \omega^{1/2} \tilde{E}_{\eta} \Vert_{K}\\
\le & C \tilde{M}_K
  h^2_{K} \left( \Vert  (\omega^{-1}G)_{\beta\beta} \Vert_{K}+
\Vert  (\omega^{-1}G)_{\beta\eta} \Vert_{K} +\Vert  (\omega^{-1}G)_{\eta\eta} \Vert_{K} \right).
\end{split}
\end{equation}
Substituting \eqref{eq:omega-G-xx}, \eqref{eq:omega-G-yy} and \eqref{eq:omega-G-xy} into \eqref{interpolation of Ex and Ey},  we have
\begin{align}\label{Ex-Ey}
(\varepsilon+b^{2}\delta)\Vert \omega^{1/2}\tilde{E}_{\beta} \Vert^{2}+\varepsilon\Vert\omega^{1/2} \tilde{E}_{\eta} \Vert^{2}
&\le
Ck^{-2}\vvvert G \vvvert^{2}_{\omega}.
\end{align}
More precisely, we have
\begin{equation}\label{eq:Ex-Omega-x}
\begin{split}
\Vert \omega^{1/2}\tilde{E}_{\beta} \Vert_{ \Omega_x }
\le &
Ck^{-1}N^{-1}( \sigma^{-1}_{\beta}\varepsilon^{-1/2}   +\sigma^{-2}_{\eta}+\sigma^{-1}_{\eta} \varepsilon^{-1/2})\vvvert  G \vvvert_{\omega}\\
\le &
Ck^{-1} \varepsilon^{-1/2} \ln^{-1}N  \vvvert  G \vvvert_{\omega}.
\end{split}
\end{equation}

Substituting \eqref{eq:omega-G-xx}, \eqref{eq:omega-G-yy} and \eqref{eq:omega-G-xy} into \eqref{interpolation of Ex and Ey} again,  we have
\begin{equation}\label{E Omegas Omegay}
\Vert \omega^{1/2}\tilde{E}\Vert_{K}
\le
\left\{
\begin{array}{ll}
Ck^{-1}N^{-1/2}\vvvert  G \vvvert_{\omega,K}  &\text{if $K\subset\Omega_s$}\\
Ck^{-1}\varepsilon^{1/2}\vvvert  G \vvvert_{\omega,K}  &\text{if $K\subset\Omega_{xy}$}
\end{array}
\right.
,
\end{equation}
and  
\begin{equation}\label{E Omega-xy}
\Vert \omega^{1/2}\tilde{E}\Vert_{K}
\le
Ck^{-1} \varepsilon^{-1/2} N^{-1} \ln^{-1}N\vvvert  G \vvvert_{\omega,K}\;\;
\text{if $K\subset\Omega_x\cup\Omega_y$}.
\end{equation}

For what follows we need a sharper bound of $\Vert \omega^{1/2}\tilde{E} \Vert_{\Omega_x\cup\Omega_y }$.       Considering $\sigma_{\beta}\ge kN^{-1}\ln N$, $\sigma_{\eta}\ge k N^{-1/2}$, \eqref{E Omega-xy} and  \eqref{eq:Ex-Omega-x}, similar to \cite[Lemma 4.4]{Linb1Styn2:2001-SDFEM} we obtain  
\begin{equation}\label{E Omegax}
\begin{split}
&\Vert \omega^{1/2}\tilde{E} \Vert^{2}_{\Omega_x \cup\Omega_y }
\le
C\lambda^{2}_{x}
\left\{\Vert(\omega^{1/2})_{\beta}\tilde{E}\Vert^{2}_{\Omega_x \cup\Omega_y }+\Vert\omega^{1/2}\tilde{E}_{\beta}\Vert^{2}_{\Omega_x \cup\Omega_y }\right\}\\
\le &
C\varepsilon^{2}\ln^{2}N\cdot
\left\{\sigma^{-2}_{\beta}\Vert \omega^{1/2}\tilde{E} \Vert^{2}_{\Omega_x \cup\Omega_y }+\Vert\omega^{1/2}\tilde{E}_{\beta}\Vert^{2}_{\Omega_x\cup\Omega_y }\right\}\\
\le  &
Ck^{-2}\varepsilon^{2}\ln^{2}N
    \{\sigma^{-2}_{\beta} \varepsilon^{-1 } N^{-2} \ln^{-2}N +
\varepsilon^{-1 } \ln^{-2}N 
    \}\vvvert G \vvvert^{2}_{\omega}\\
\le     &
    Ck^{-2}\varepsilon  \vvvert G \vvvert^{2}_{\omega}.
    \end{split}
\end{equation}

\par
Substituting \eqref{Ex-Ey}, \eqref{E Omegas Omegay}   and  \eqref{E Omegax} into \eqref{B(E,G)} and recalling the definition
of $\delta$, we obtain
\begin{equation*}
|a_{SD}(\tilde{E},G)|\le
Ck^{-1}\vvvert G \vvvert^{2}_{\omega}.
\end{equation*}
Choosing $k$ sufficiently large independently
of $\varepsilon$ and $N$, we are done.
\end{proof}

Lemmas 
\ref{lem:1}, \ref{lem:2} and \ref{lem:3}
 yield the following bound of the discrete Green function in the energy norm just as in \cite[Theorem 4.1]{Linb1Styn2:2001-SDFEM}.
\begin{theorem}\label{theorem:energy estimate}
Assume  that $\sigma_{\beta}= kN^{-1}\ln N$ and $\sigma_{\eta}= k N^{-1/2}$ in \eqref{eq:Omega0}, where
$k$ is chosen so  that Lemmas \ref{lem:1}, \ref{lem:2} and \ref{lem:3} hold. Then for
$\boldsymbol{x}^{\ast}\in \Omega\setminus\Omega_{xy}$, we have
$$
 \vvvert G \vvvert  \le \sqrt{8}\vvvert G \vvvert_{\omega} \le 
CN^{1/2}\ln^{1/2} N.
$$
\end {theorem}

\begin{remark}
For different crosswind diffusion coefficients, for example $\hat{\varepsilon}$ in 
\cite{Linb1Styn2:2001-SDFEM} and $\varepsilon$ in the present paper, we can choose $\sigma_{\beta}$ and $\sigma_{\eta}$ \emph{large} enough such that Lemmas \ref{lem:1}--\ref{lem:3} hold true. Thus with different assumptions on these parameters and  different SDFEMs, we can obtain bounds similar to \cite[Theorem 4.1]{Linb1Styn2:2001-SDFEM}, i.e., Theorem \ref{theorem:energy estimate}.
\end{remark}


%










\end{document}